\documentclass[11pt,reqno,tbtags]{amsart}
\pdfoutput=1
\usepackage[]{amsmath,amssymb,amsfonts,latexsym,amsthm,enumerate}
\usepackage{amssymb}
\usepackage{enumerate,graphicx,subfigure,paralist}
\usepackage[numeric,initials,nobysame]{amsrefs}

\numberwithin{equation}{section}

\newtheorem{maintheorem}{Theorem}

\newtheorem{theorem}{Theorem}[section]
\newtheorem*{theorem*}{Theorem}

\newtheorem*{conjecture*}{Conjecture}
\newtheorem{lemma}[theorem]{Lemma}
\newtheorem{claim}[theorem]{Claim}

\newtheorem{corollary}[theorem]{Corollary}

\theoremstyle{definition}{

\newtheorem*{definition*}{Definition}

}

\theoremstyle{remark}{

\newtheorem*{remark*}{Remark}

}

\newcommand{\R}{\mathbb R}

\newcommand{\deq}{\stackrel{\scriptscriptstyle\triangle}{=}}

\newcommand{\E}{\mathbb{E}}
\renewcommand{\P}{\mathbb{P}}
\DeclareMathOperator{\var}{Var}

\newcommand{\gap}{\text{\tt{gap}}}
\newcommand{\tmix}{t_\textsc{mix}}

\renewcommand{\epsilon}{\varepsilon}
\renewcommand{\phi}{\varphi}

\newcommand{\cP}{\mathcal{P}}
\newcommand{\cL}{\mathcal{L}}

\newcommand{\cT}{\mathcal{T}}

\DeclareMathOperator{\dist}{dist}
\DeclareMathOperator{\diam}{diam}
\DeclareMathOperator{\vol}{vol}
\newcommand{\SRW}{\textsf{SRW}}

\newcommand{\iso}{\operatorname{ch}}
\date{}

\begin{document}
\title{Explicit expanders with cutoff phenomena}

\author{Eyal Lubetzky}
\address{Eyal Lubetzky\hfill\break
Microsoft Research\\
One Microsoft Way\\
Redmond, WA 98052, USA.}
\email{eyal@microsoft.com}
\urladdr{}

\author{Allan Sly}
\address{Allan Sly\hfill\break
Microsoft Research\\
One Microsoft Way\\
Redmond, WA 98052, USA.}
\email{allansly@microsoft.com}
\urladdr{}

\begin{abstract}
The cutoff phenomenon describes a sharp transition in the convergence of an ergodic finite Markov chain to equilibrium.
Of particular interest is understanding this convergence for the simple random walk on a bounded-degree expander graph.
The first example of a family of bounded-degree graphs where the random walk exhibits cutoff in total-variation was provided only very recently, when the authors showed this for a typical random regular graph.
However, no example was known for an explicit (deterministic) family of expanders with this phenomenon.
Here we construct a family of cubic expanders where the random walk from a worst case initial position exhibits total-variation cutoff. Variants of this construction give cubic expanders without cutoff, as well as cubic graphs
with cutoff at any prescribed time-point.
\end{abstract}

\maketitle

\vspace{-1cm}

\section{Introduction}\label{sec:intro}

A finite ergodic Markov chain is said to exhibit \emph{cutoff} in total-variation if its $L^1$-distance from the stationary distribution drops abruptly from near its maximum to near $0$. In other words, one should run the Markov chain until the cutoff point for it to even slightly mix in $L^1$ whereas running it any further is essentially redundant.

Let $(X_t)$ be an aperiodic irreducible discrete-time Markov chain on a finite state space $\Omega$ with stationary distribution $\pi$. The worst-case total-variation distance to stationarity at time $t$ is defined as
\[ d(t) \deq \max_{x \in \Omega} \| \P_x(X_t \in \cdot)- \pi\|_\mathrm{TV}\,,\]
where $\P_x$ denotes the probability given $X_0=x$ and where $\|\mu-\nu\|_\mathrm{TV}$, the \emph{total-variation distance} of two distributions $\mu,\nu$ on $\Omega$, is given by
\[\|\mu-\nu\|_\mathrm{TV} \deq \sup_{A \subset\Omega} \left|\mu(A) - \nu(A)\right| = \frac{1}{2}\sum_{x\in\Omega} |\mu(x)-\nu(x)|\,.\]
Define $\tmix(\epsilon)$, the total-variation \emph{mixing-time} of $(X_t)$ for $0 < \epsilon < 1$, to be
\[ \tmix(\epsilon) \deq \min\left\{t : d(t) < \epsilon \right\}\,.\]
Let $\big(X_t^{(n)}\big)$ be a family of such chains, each with its total-variation distance from stationarity $d_n(t)$, its mixing-time $\tmix^{(n)}$, etc. This family exhibits \emph{cutoff} iff the following sharp transition in its convergence to equilibrium occurs:
\begin{equation}\label{eq-cutoff-def}\lim_{n\to\infty} \tmix^{(n)}(\epsilon) \big/ \tmix^{(n)}(1-\epsilon)=1 \quad\mbox{ for any $0 < \epsilon < 1$}\,.\end{equation}
The rate of convergence in~\eqref{eq-cutoff-def} is addressed by the following notion of a \emph{cutoff window}: For two sequences $t_n,w_n$ with $w_n=o(t_n)$ we say that $\big(X_t^{(n)}\big)$ has cutoff at time $t_n$ with window $w_n$ if and only if
\[\tmix^{(n)}(s) = \left(1+O(w_n)\right)t_n = (1+o(1))t_n~\mbox{ for any fixed $0 < s < 1$}\,,\]
or equivalently, cutoff at time $t_n$ with window $w_n$ occurs if and only if
\[\left\{\begin{array}
  {l}\lim_{\lambda\to\infty} \liminf_{n\to\infty} \;d_n(t_n - \lambda w_n) = 1\,,\\
  \lim_{\lambda \to \infty} \limsup_{n\to\infty} d_n(t_n + \lambda w_n) = 0\,.
\end{array}\right.\]

The cutoff phenomenon was first identified for random transpositions on the symmetric group in~\cite{DS} and for random walks on the hypercube in~\cite{Aldous}.
The term ``cutoff''
was coined by Aldous and Diaconis in~\cite{AD}, where cutoff was shown for the top-in-at-random card shuffling process.
While believed to be widespread, there are relatively few examples where the cutoff phenomenon has been
rigorously confirmed. Even for fairly simple chains, determining whether there is cutoff often requires
the full understanding of their delicate behavior around the mixing threshold.
See~\cites{Diaconis,CS,SaloffCoste} and the references therein for more on the cutoff phenomenon.

A specific Markov chain which found numerous applications in a wide range of areas in mathematics over the last quarter of a century is the simple random walk (\SRW) on a bounded-degree \emph{expander} graph. A finite graph is called an expander
if every small subset of the vertices has a relatively large edge boundary. Formally, the Cheeger constant of a $d$-regular graph $G$ on $n$ vertices (also referred to as the edge isoperimetric constant) is defined as
\[ \iso(G) = \min_{\emptyset \neq S \subsetneqq V(G)} \frac{|\partial S|}{|S|\;\wedge\; |V(G)\setminus S|}\,,\]
where $\partial S$ is the set of edges with exactly one endpoint in $S$. We say that $G$ is a $c$-edge-expander for some fixed $c > 0$ if it satisfies $\iso(G) > c$. The well-known discrete analogue of Cheeger's inequality~\cites{Alon,AM,Dodziuk,JS} implies that the spectral-gap of the \SRW\ on a family of $c$-edge-expander graphs on $n$ vertices is uniformly bounded away from $0$, hence these chains rapidly converge to equilibrium within $O(\log n)$ steps. See the survey~\cite{HLW} for more on the applications of random walks on expanders.

In 2004, Peres\cite{Peres} observed that for any family of reversible Markov chains, total-variation cutoff can only occur if the inverse spectral-gap has smaller order than the mixing time. Note that this condition clearly holds for the simple random walk on an $n$-vertex expander, where the inverse-gap is $O(1)$ whereas $\tmix \asymp \log n$. It was shown by Chen and Saloff-Coste \cite{CS} that when measuring convergence in $L^p$-distance for $p>1$ this criterion does ensure cutoff, however the case $p=1$ (cutoff in total-variation)
has proved to be significantly more complicated. There are known examples where the above condition does not imply cutoff (see \cite{CS}*{Section 6}), yet it was conjectured by Peres to be sufficient in many natural families of chains (e.g.~\cite{DLP} confirming this for birth-and-death chains). In particular, this was conjectured for the lazy random walk on bounded-degree transitive graphs.

The first example of a family of bounded-degree graphs where the random walk exhibits cutoff in total-variation was provided only very recently \cite{LS}, when the authors showed this for a typical random regular graph.
It is well known that for any fixed $d \geq 3$, a random $d$-regular graph is with high probability (w.h.p.) a very good expander, hence the simple random walk on almost every $d$-regular expander exhibits worst-case total-variation cutoff.
However, to this date there were no known examples for an \emph{explicit} (deterministic) family of expanders with this phenomenon.

In Section~\ref{sec-explicit} we provide what is, to the best of our knowledge, the first explicit construction of a
family of bounded-degree expanders where the simple random walk has worst-case total-variation cutoff.
\begin{maintheorem}\label{thm-1}
There is an explicit family of $3$-regular expanders on which the \emph{\SRW}\ from a worst case initial position exhibits total-variation cutoff.
\end{maintheorem}
The construction mimics the behavior of the \SRW\ on random regular graphs, whose mixing was shown in \cite{LS} (as conjectured by Durrett~\cite{Durrett} and Berestycki~\cite{Berestycki}) to resemble that of a walk started at a root of a $d$-regular tree. Two smaller expanders that are embedded into the graph structure
allow careful control over the mixing time from all possible initial positions.

A straightforward modification of the above construction yields an explicit family of cubic expanders where the \SRW\ from a worst-case initial position does \emph{not} exhibit cutoff in total-variation (despite Peres' cutoff criterion).
Note that Peres and Wilson \cite{PW} had already sketched an example for a family of expanders without total-variation cutoff.
We describe our simple construction achieving this in Section~\ref{sec-nocutoff} for completeness.

A final variant of the construction, presented in Section~\ref{sec-prescribed}, provides cubic graphs with cutoff occurring at essentially any prescribed order of location.
Namely, there is an explicit family of cubic graphs where the \SRW\ has cutoff at any specified order between $[\log n, n^2)$ whereas for $\tmix \asymp n^2$ there cannot be cutoff (it is well-known that
on any family of bounded-degree graphs on $n$ vertices the \SRW\ has $c\log n \leq \tmix \leq c' n^2$ for some fixed $c,c'>0$).
\begin{maintheorem}
  \label{thm-2}
Let $t_n$ be a monotone sequence with $t_n \geq \log n$ and $t_n = o(n^2)$. There is an explicit family $(G_n)$ of $3$-regular graphs with $|G_n|\asymp n$ vertices where the \emph{\SRW}\ from a worst-case initial position exhibits total-variation cutoff at $\tmix \asymp t_n$.

Furthermore, for any family of bounded-degree $n$-vertex graphs where the \emph{\SRW} has $\tmix \asymp n^2$ (largest possible order of mixing) there cannot be cutoff.
\end{maintheorem}

\section{Explicit constructions achieving cutoff}

\subsection{Proof of Theorem~\ref{thm-1}: explicit expanders with cutoff}\label{sec-explicit}
To simplify the exposition, we will first construct a family of 5-regular expanders where the \SRW\ from a worst initial position exhibits cutoff. Subsequently, we will describe how to modify the construction to yield a family of cubic expanders with this property (as per the statement of Theorem~\ref{thm-1}).

The graph we construct will contain a smaller explicit expander on a fixed proportion of its vertices, connected to what is essentially a product of another expander with a ``stretched'' regular tree (one where the edges in certain levels are replaced by paths).

Let $h\to\infty$ and let $H_1,H_2$ be two explicit expanders as follows (cf.\ e.g.~\cite{Ajtai} for an explicit construction of a $3$-regular graph, as well as \cite{RVW} and the references therein for additional explicit constructions of constant-degree expander graphs):
\begin{itemize}
  \item $H_1$ : An explicit $3$-regular expander on $20\cdot 2^{2h}$ vertices.
  \item $H_2$ : An explicit $4$-regular expander on $20 \cdot 2^{6h}$ vertices.
\end{itemize}
Let $\lambda(H_i)$ denote the largest absolute-value of any nontrivial eigenvalue of $H_i$ for $i=1,2$.
Finally, let $L$ be some sufficiently large fixed integer whose value will be specified later.

Our final construction for $5$-regular expander will be based on a (modified) regular tree, hence it will be convenient to describe its structure according to the various levels of its vertices. Let the vertex $\rho$ denote the root of the tree, and construct the graph $G$ as follows:
\begin{enumerate}[\quad 1.]
  \item Levels $0,1,2$: First levels of a $5$-regular tree rooted at $\rho$.
  \begin{list}{\labelitemi}{\leftmargin=1em}
    \item Denote by $U=\{u_1,\ldots,u_{20}\}$ the vertices comprising level $2$.
  \end{list}
  \item \label{cons-part-1} Levels $3,\ldots,h+2$: Stretched $4$-ary trees rooted at $U$:
  \begin{list}{\labelitemi}{\leftmargin=1em}
    \item For each $u_i\in U$ place an $h$-level $4$-ary tree $\cT_{u_i}$ rooted at $u_i$
and identify the vertices of $\cT_{u_i}$ and $\cT_{u_j}$ via the trivial isomorphism.

  \item \label{cons-part-2} Replace every edge of each $\cT_{u_i}$ by a (disjoint) path of length $L$.

    Connect $\cT_{u_1}^*$, the new interior vertices in $\cT_{u_1}$ (with initial degree $2$)
to their isomorphic counterparts in $\cT_{u_2}^*,\cT_{u_3}^*,\cT_{u_4}^*$ (add $4$-cliques between identified interior vertices) and similarly for $\{\cT_{u_5}^*,\ldots,\cT_{u_{8}^*}\}$ etc.
    \item Let $A$ denote the final $20\cdot 4^h$ vertices comprising level $h+2$, and associate the vertices of $A$ with those of $H_1$.
  \end{list}
  \item \label{cons-part-3}
  Levels $h+3,\ldots,2h+2$: Product of $H_1$ and a stretched $4$-ary tree.
   \begin{list}{\labelitemi}{\leftmargin=1em}
    \item For each $a \in A$ place an $h$-level $L$-stretched $4$-ary tree $\cT_a$.

     Connect vertices in $\cT_a^*$ with their counterparts in $\cT_b^*$ for $ab \in E(H_1)$.
    \item Let $B$ denote the final $20\cdot 4^{2h}$ vertices comprising level $2h+2$.
  \end{list}
  \item Levels $2h+3,\ldots,3h+2$: A forest of $4$-ary trees rooted at $B$.
  \item Last level: Associate leaves with $H_2$ and interconnect them accordingly.
\end{enumerate}
\begin{figure}
\centering \includegraphics[width=4.5in]{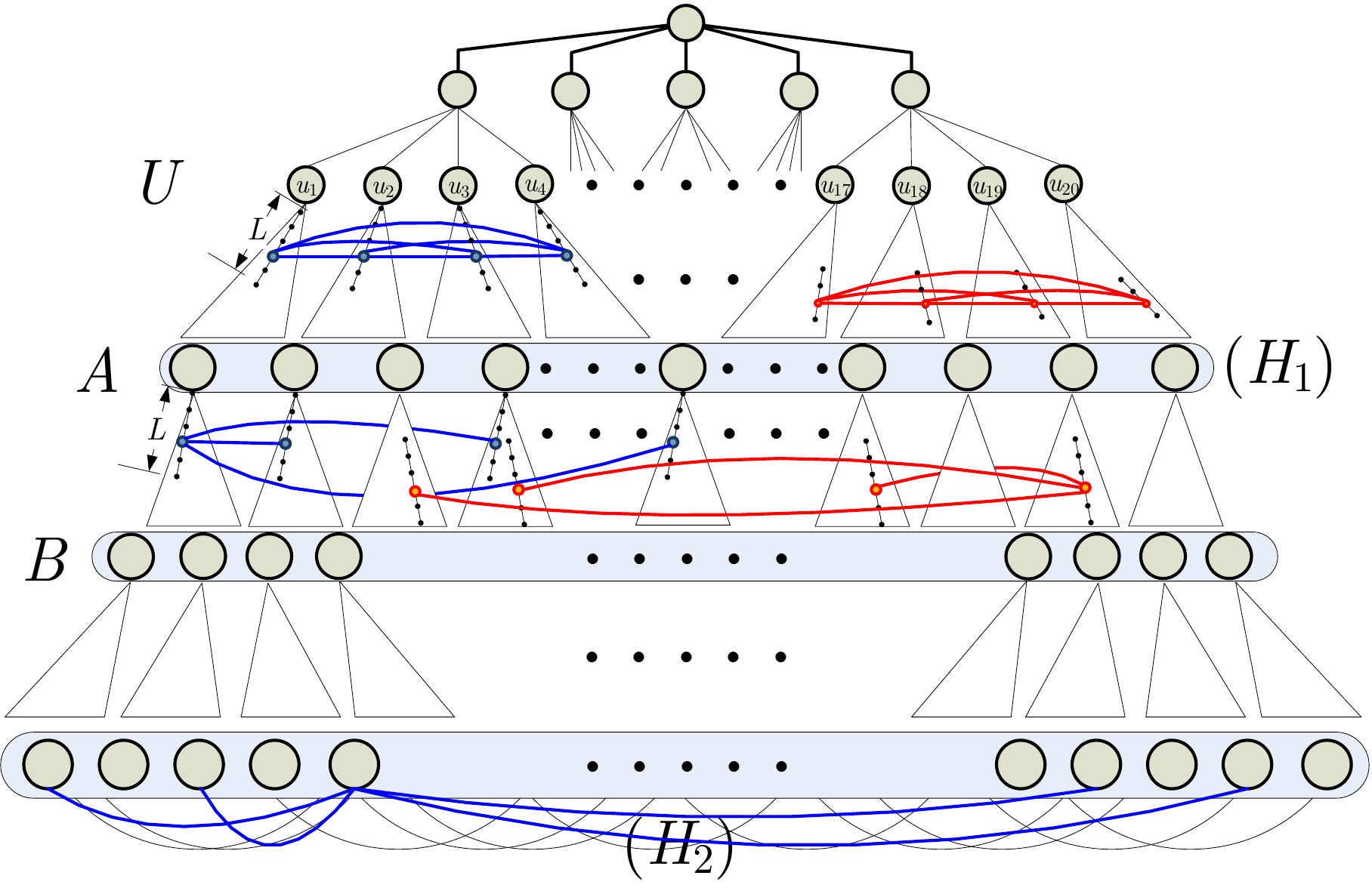}
\caption{Explicit construction of a $5$-regular graph on which the random walk exhibits
total-variation cutoff.}
\label{fig-expcons}
\end{figure}
Finally, the aforementioned parameter $L$ is chosen as follows: Denote by
\[\gap_1=\inf_{|H_1|}\left(1-\lambda(H_1)/3\right)~,~\gap_2=\inf_{|H_2|}\left(1-\lambda(H_2)/4\right)\] the minimum spectral-gaps in the explicit expanders that were embedded in our construction (recalling from the introduction that $\gap_i > 0$ for both $i=1,2$ by the definition of expanders together with the discrete analogue of Cheeger's inequality), and define
\begin{equation}
   \label{eq-L-choice}
   L = \bigg\lceil\frac{2}{\sqrt{\gap_1}} \;\vee\; \frac{16}{\gap_2}\;\vee\; 32\bigg\rceil\,.
 \end{equation}
See Fig.~\ref{fig-expcons} for an illustration of the above construction.

We begin by establishing the expansion of the above constructed $G$. Throughout the proof we omit ceilings and floors in order to simplify the exposition.
\begin{lemma}\label{lem-exp}
Let $\kappa = (\iso(H_2)\wedge 1)/3$ for $H_2$ our explicit $4$-regular expander.
For any integer $L > 0$, the Cheeger constant of the above described $5$-regular graph $G$ with parameter $L$ satisfies $\iso(G) \geq \kappa/25L$. Moreover, the induced subgraph $\tilde{G}$ on the last $h$ levels (i.e., levels $2h+2,\ldots,3h+2$)
has $\iso(\tilde{G}) \geq \kappa$.
\end{lemma}
\begin{proof}
First consider the entire graph $G$. Since we are only interested in a lower bound on $\iso(G)$, clearly it is valid to omit edges from the graph, in particular we may erase the cross edges between any subtrees $\cT_{u}^*,\cT_{v}^*$ described in Items~\ref{cons-part-2},\ref{cons-part-3} of the construction. This converts every stretched edge of the $5$-regular tree of $G$ simply into a $2$-path (one where all interior vertices have degree 2) of length $L$.

Next, we contract all the above mentioned $2$-paths into single edges and denote the resulting graph by $F$. The next simple claim shows that this decreases the Cheeger constant by at most $O(L)$.

\begin{claim}
Let $F$ be a connected graph with maximal degree $\Delta$ and let $G$ be a graph on at most $\frac32|F|$ vertices obtained via replacing some of the edges of $F$ by $2$-paths of length $L$. Then $\iso(G) \geq \iso(F) /\Delta^2 L$.
\end{claim}
\begin{proof}
Let $X\subset V(G)$ be a set of cardinality at most $|G|/2$ achieving the Cheeger constant of $G$.
We may assume that $\Delta \geq 3$ otherwise $F$ is a disjoint union of paths and cycles and the result holds trivially.

Notice that if $X$ contains two endpoints of a $2$-path $\cP=(x_0,\ldots,x_L)$ while only containing $k < L-1$ interior vertices of $\cP$ then we can assume that $X \cap \cP = \{x_0,\ldots,x_k,x_L\}$, i.e.\ all the interior vertices are adjacent (this maintains the same cardinality of $X$ while not increasing $\partial X$).
With this in mind, modify the set $X$ into the set $X'$ by repeating the following operation: As long as there is a 2-path $\cP$ as above (with $x_0,\ldots,x_k$ and $x_L$ in $X$ for some $k < L-1$) we replace $x_L$ by $x_{k+1}$.
This maintains the cardinality of the set while increasing its edge-boundary by at most $\Delta-2$ (as $x_L$ formerly contributed at least $1$ edge to this boundary due to $x_{L-1}\notin X$).
Altogether, this yields a set $X'$ where no $2$-path $\cP=(x_0,\ldots,x_L)\nsubseteq X'$ has both $x_0,x_L\in X'$, while $X'$ satisfies
\[ |\partial X'|/|X'| \leq (\Delta-1)\iso(G)\,.\]
The obtained subset $X'$ is possibly disconnected, and we will next argue that its connected components satisfy an appropriate isoperimetric inequality.
Consider $X''$, the connected component of $X'$ that minimizes $|\partial X''|/|X''|$.
If $X''$ is completely contained in the interior of one of the new $2$-paths then the statement of the claim immediately holds since
\[ \iso(G) \geq \frac{|\partial X'|}{(\Delta-1)|X'|}  \geq \frac{|\partial X''|}{(\Delta-1)|X''|}\geq \frac{2}{(L-1)(\Delta-1)} \geq \frac{\iso(F)}{\Delta^2 L}\,,\]
with the last inequality due to the fact $\iso(F) \leq \Delta$.
Suppose therefore that this is not the case hence we may now assume that $X''$ contains at least one endpoint of any $2$-path it intersects.

Let $Y = X'' \cap V(F)$, i.e.\ the subset of the vertices of $F$ obtained from $X''$ by excluding any vertex that was created in $G$ due to subdivision of edges.
Observe that our assumption on $X'$ implies that
\[|\partial Y| = |\partial X''|\,,\] since either a $2$-path $\cP$ is completely contained in $X''$ (not contributing to $\partial X''$) or $\cP \cap X = \{x_0,\ldots,x_k\}$ for some $k < L$ (contributing the edge $x_k,x_{k+1}$ to $\partial X''$, corresponding to the edge $x_0,x_L$ in $\partial Y$).

It remains to consider $|Y|$. Clearly, $X''$ can be obtained from $Y$ by adding at most $\Delta$ new 2-paths with $L-1$ new interior points per vertex, hence
\[ |Y| \geq |X''|/\Delta L \,.\]
On the other hand, since $|Y|\leq |X''|\leq |G|/2$ and $|G| \leq \frac32 |F|$ we have
\[ |Y|/|F| \leq \tfrac32 |X''|/|G| \leq \tfrac34\,,\]
which together with the fact that $|X''|\leq |G|/2$ implies that
\[ |V(F) \setminus Y | \geq \tfrac14 |F| \geq \tfrac16 |G| \geq \tfrac13 |X''| \geq |X''|/ \Delta L\,.\]
Altogether,
\begin{equation*}
\iso(F) \leq \frac{|\partial Y|}{|Y|\;\wedge\; |V(F)\setminus Y|}
\leq \Delta L \frac{|\partial X''|}{|X''|} \leq \Delta^2 L\, \iso(G)\,.
\qedhere
\end{equation*}
\end{proof}
In light of the above claim we have $\iso(G) \geq \iso(F)/25 L$ where the graph $F$ is the result of taking a complete $5$-regular tree of height $3h+2$ levels and connecting its $5\cdot 4^{3h+1}$ leaves, denoted by $F'$, via the $4$-regular expander $H_2$.
It therefore remains to show that $\iso(F) \geq \kappa$.

Let $S$ be a set of size $s \leq |F|/2$ vertices that achieves $\iso(F)$. Define its subset of the leaves $S' = S\cap F'$ and set $s'=|S'|$.
Since $|F'| \geq \frac34 |F|$ we clearly have $s' \leq s \leq \frac23 |F'|$ hence
$\left(|S'| \wedge |F' \setminus S'|\right) \geq s'/2$. We thus have the following two options:
\begin{enumerate}
  \item $s' \geq \frac23 s$: In this case
  \[|\partial_F S| \geq |\partial_{F'} S'| \geq \iso(H_2) s'/2 \geq \iso(H_2)s/3\,.\]
  \item $s' < \frac23s$: Letting $\mathbb{T}_5$ denote the infinite $5$-regular tree (whose Cheeger constant equals 3) we get
  \[ |\partial_{F} S| \geq \iso(\mathbb{T}_5)(s - s')-s' = 3s - 4s' > s/3\,.\]
\end{enumerate}
Altogether we deduce that
\[ \iso(G) \geq \iso(F) \geq (\iso(H_2)\wedge 1)/3 = \kappa\,.\]

The second part of the lemma (the statement on the subgraph $\tilde{G}$) follows from essentially the same argument given above for $\iso(F)$, as $\tilde{G}$ is precisely a forest of $5$-regular trees of height $h$ where all the leaves are connected via the expander $H_2$. Again we get $\iso(\tilde{G})\geq \kappa$, completing the proof.
\end{proof}

Consider the \SRW\ started $k\leq h$ levels above the bottom (i.e.\ at level $3h+2-k$) of the graph $G$. The height of the walk is then a one-dimensional biased random walk with positive speed $\frac35$, implying that it would reach the bottom after $\frac53 k + o(h)$ steps with high probability.

On the other hand, if the \SRW\ is started closer to the root, i.e.\ at level $2h+2-k$, then the one-dimensional random walk is delayed by two factors, horizontal (cross-edges) and vertical (stretching the edges into paths). Until reaching level $2h+2$ (after which the previous analysis applies), these delays are encountered along $\frac53k+o(h)$ stretched edges with the following effect:
\begin{compactitem}
  \item The former incurs a laziness delay with probability $\frac35$ whenever the walk is positioned on an interior vertex of a 2-path.
    \item The latter delays the walk by the passage time of a \SRW\ through an $L$-long 2-path, where the walk leaves the origin with probability $1$.
\end{compactitem}
It is well-known (and easy to derive) that the expected passage time of the one-dimensional \SRW\ from $0$ to $\pm L$ is precisely $L^2$ and the expected number of visits to the origin by then (including the starting position) is exactly $L$. It thus follows that the expected delay of the one-dimensional walk representing our height in the tree along a single stretched edge is
\[\tfrac52 (L^2 - L) + L = \tfrac12L(5 L - 3)\,.\]

Combining the above cases we arrive at the following conclusion:
\begin{claim}\label{clm-tau-ell}
Consider the \emph{\SRW} on the graph $G$ started at a vertex on level $s\in\{0,\ldots,3h+1\}$. Set $\alpha= s/h$  and let $\tau_\ell$ be the hitting time of the walk to the leaves (i.e.\ to level $3h+2$). Then w.h.p.
\[ \tau_\ell = \left\{\begin{array}{ll}
  (\frac53+o(1))\left[L(5 L - 3) (1-\frac{\alpha}2) + 1\right]  h & \mbox{If $0 \leq \alpha \leq 2$}\,,\\
  \frac53 (3-\alpha) h+o(h) & \mbox{If $\alpha \geq 2 $}\,.
\end{array}\right.\]
\end{claim}

The next lemma relates $\tau_\ell$, the hitting time to the leaves (addressed by the above claim), and the mixing of the \SRW\ on the graph.
\begin{lemma}\label{lem-tmix-upper-bound-from-top}
Let $\epsilon > 0$, let $s_0$ be some vertex on level $l_0\in\{0,\ldots,h+2\}$ and $T=(1+\delta)\E_{s_0}\tau_\ell$ for $\delta >0$ fixed, where $\tau_\ell$ is the hitting time of the \emph{\SRW} to the leaves. Then
$ \|\P_{s_0}(S_{T} \in \cdot) - \pi \|_\mathrm{TV} < \epsilon$ for any sufficiently large $h$.
\end{lemma}
\begin{proof}
Let $(S_t)$ denote the \SRW\ started at some vertex $s_0$ in level $l_0 \leq h+2$ and $\pi$ be the uniform distribution on $V(G)$. Let $(\tilde{S}_t)$ be a random walk started from the uniform distribution $\tilde{S}_0 \sim \pi$.
Write $\cL_i$ for $i\in\{0,\ldots,3h+2\}$ for the vertices of level $i$ in $G$ (accounting for all the vertices except interior ones along the 2-paths of length $L$ corresponding to stretched edges) and let
 $\psi:G\to \{0,\ldots,3h+2\}$ map vertices in the graph to their level (while mapping interior vertices of $2$-paths to the lower of their endpoint levels).

Further let $\Omega = \{2h+3,\ldots,3h+2\}$. Clearly, for large enough $h$ we have
\[ \P\big(\psi(\tilde{S}_0) \notin \Omega\big) < \frac{\epsilon}{10}\,.\]
Furthermore, due to the bias of the \SRW\ towards the leaves and the fact that $\tau_\ell=(1+o(1))\E \tau_\ell \asymp h$ (recall Claim~\ref{clm-tau-ell} and that $L$ is fixed) we deduce that
$\psi(S_{T}) > \frac52h$ except with probability exponentially small in $h$, and in particular for any sufficiently large $h$
\[ \P\big(\psi(S_T) \notin \Omega\big) < \frac{\epsilon}{10}\,.\]
Therefore, an elementary calculation shows that if $\psi(S_T)$ and $\psi(\tilde{S}_T)$ are close in total-variation and so are $S_T$ given $\psi(S_T)=i$ and $\tilde{S}_T$ given $\psi(\tilde{S}_T)=i$ for all $i\in\Omega$, then the required statement on $\|\P(S_T\in\cdot)-\pi\|$ would follow.
Namely, if we should show that at time $t=T$ we have
\begin{align}
\label{eq-tv-in-level}\left\|\P\big(S_t \in \cdot \mid \psi(S_t)=i\big) - \P\big(\tilde{S}_t \in \cdot \mid \psi(\tilde{S}_t)=i\big)\right\|_\mathrm{TV}&<\frac{\epsilon}5
\mbox{ for all $i\in\Omega$}\,,\\
\label{eq-tv-level}\left\|\P\big(\psi(S_t)\in\cdot\big) - \P\big(\psi(\tilde{S}_t) \in \cdot \big)\right\|_\mathrm{TV}&<\frac{\epsilon}5\,,
\end{align}
then we would get that $\|\P(S_{t} \in \cdot) - \pi \|_\mathrm{TV} < \epsilon$ (with room to spare).

Examine the period spent by $(S_t)$ in levels $\{h+2,\ldots,2h+2\}$. The graph in these levels is essentially a product of a $4$-ary tree whose edges are stretched into $L$-long $2$-paths and the expander $H_1$.
Let $\phi:G \to H_1$ map the vertices in these levels to their corresponding vertices in $H_1$, and let $\tau_0,\tau_1$ be the hitting times of $(S_t)$ to levels $\frac32 h$ and $2h+2$ respectively.

As argued above, the \SRW\ started at level $2h+2-k$ is a one-dimensional biased random walk that w.h.p.\ passes through $(1+o(1))\frac53 k+o(h)$ stretched edges until reaching level $2h+2$ for the first time. In particular, between times $\tau_0,\tau_1$ the walk w.h.p. passes through $(1+o(1))\frac56 h$ stretched edges.

Along each stretched edge among levels $\{h+2,\ldots,2h+2\}$, the walk traverses a cross-edge in $H_1$ (that is, $\phi(S_{t+1})$ is uniformly distributed over the neighbors of $\phi(S_t)$ in $H_1$) with probability $\frac35$ whenever it is in an interior vertex in the $2$-path, for a total expected number of $\tfrac32 (L^2 - L)$ such moves.

Finally, due to its bias towards the leaves, with high probability the \SRW\ from level $\frac32h$ reaches level $2h+2$ (the vertices $B$) before hitting level $h+2$.
Applying CLT we conclude that the \SRW\ w.h.p.\ traverses
\[ (\tfrac54 + o(1))L(L-1) h > L^2 h\]
cross-edges (each corresponding to a single step of the \SRW\ on $H_1$) between times $\tau_0,\tau_1$, where the last inequality holds for $L>5$ and large enough $h$. This amounts to at least $L^2 h$ consecutive steps of a \SRW\ along $H_1$.

Aiming for a bound on the total-variation mixing, we may clearly condition on events that occur with high probability:
Condition therefore throughout the proof that indeed the above statement holds.

In particular, letting $(X_t)$ be the \SRW\ on the expander $H_1$ and recalling that $|A| = |H_1| = 20\cdot 2^{2h}$ and $r = L^2 h$
it follows that
\begin{align*}
\max_{s_0}\left\| \P_{s_0}\left(\phi(S_{\tau_1}) \in \cdot\right) - |A|^{-1} \right\|_\mathrm{TV}
&\leq \max_{x_0}\left\| \P_{x_0}\left(X_r \in \cdot \right) - |A|^{-1} \right\|_\mathrm{TV}  \\
&\leq \frac12 \max_{x_0}\left\| \P_{x_0}\left(X_r \in \cdot \right) - |A|^{-1} \right\|_{2} \,.
\end{align*}
Recalling that $H_1$ is a $3$-regular with second largest (in absolute value) eigenvalue $\lambda(H_1)$ and writing $\gamma \deq 1-\frac{\lambda(H_1)}3$,
\begin{align*}
  \max_{x_0}\left\| \P_{x_0}\left(X_r \in \cdot \right) - |A|^{-1} \right\|_{2}
&\leq \sqrt{|A|} \exp\left(-\gamma r \right) \\
&< 5\exp\left[-(\gamma L^2-\log 2)h\right] < |A|^{-2}\,,
\end{align*}
with the last inequality justified for any sufficiently large $h$ provided that
\begin{equation}
  \label{eq-L-cond1}
  L > \frac{\sqrt{5\log 2}}{\sqrt{1-\lambda(H_1)/3}} \,,
\end{equation}
a fact inferred from the choice of $L$ in \eqref{eq-L-choice}. In this case, for any $s_0$
\[ \P_{s_0}\left(S_{\tau_1} = u \right) = \frac{1+o(1)}{|A|} ~\mbox{ for every $u \in A$}\,.\]
By symmetry we now conclude that for any $i \in \Omega$ and $t \geq \tau_1$ we have
\[ \P_{s_0}\left(S_t \in \cdot \mid \psi(S_t) = i\right) = \frac{1+o(1)}{|\cL_i|} ~\mbox{ for every $i \in \{2h+2,\ldots,3h+2\}$ } \,.\]
This immediately establishes~\eqref{eq-tv-in-level}.

To obtain \eqref{eq-tv-level}, note that $\tilde{S}_t$ for $t = \tau_1$ w.h.p.\ satisfies $\psi(\tilde{S}_0) > \frac52 h$. Conditioned on this event we can apply a monotone-coupling to
successfully couple $(S_t)$ and $(\tilde{S}_t)$ such that $\psi(S_{\tau_\ell})=\psi(\tilde{S}_{\tau_\ell})$, yielding~\eqref{eq-tv-level}.

The concentration of $\tau_\ell$ established in Claim~\ref{clm-tau-ell} carries the above two bounds to time $T$, thus completing the proof.
\end{proof}

Let $\tmix(\epsilon ; x)$ denote the total-variation mixing time from a given starting position $x$. That is, if $(X_t)$ is an ergodic Markov chain on a finite state space $\Omega$ with stationary distribution $\pi$ then
\[ \tmix(\epsilon; x) \deq \min\left\{t : \| \P_x(X_t \in \cdot)- \pi\|_\mathrm{TV} < \epsilon \right\}\,.\]
The above lemma gives an upper bound on this quantity for the \SRW\ started at one of the levels $\{0,\ldots,h+2\}$, which we now claim is asymptotically tight:
\begin{corollary}\label{cor-mixing-top}
Consider the \emph{\SRW} on $G$ started at some vertex $s_0$ on level $l_0\in\{0,\ldots,h+2\}$ and let $\tau_\ell$ be the hitting time of the walk to the leaves. Then for any fixed $0<\epsilon < 1$ we have $\tmix(\epsilon; s_0) = (1+o(1))\E_{s_0}\tau_\ell$.
\end{corollary}
\begin{proof}
The upper bound on $\tmix(\epsilon;s_0)$ was established in Lemma~\ref{lem-tmix-upper-bound-from-top}.

For a matching lower bound on $\tmix(1-\epsilon; s_0)$ choose some fixed integer $K = O(\log(L/\epsilon))$ such that the bottom $K$ levels of the graph comprise at least a $(1-\epsilon)$-fraction of the vertices of $G$, i.e.\
 \[ \sum_{i > 3h+2-K} \left|\cL_i\right| > (1-\epsilon)|G|\,.\]
 The lower bound now follows from observing that, by the same arguments that established Claim~\ref{clm-tau-ell}, the hitting time from level $l_0$ to level $3h+2-K$ is w.h.p.\ $(1-o(1))\E\tau_\ell$
for any sufficiently large $h$.
\end{proof}
Having established the asymptotic mixing time of the \SRW\ started at the top $h$ levels, we next wish to show that from all other vertices the mixing time is faster.
\begin{claim}\label{clm-mix-bottom}
  Let $(S_t)$ be the \emph{\SRW} started at some vertex $x$ in level $\psi(x)>h$.
For every $0 < \epsilon < 1$ and any sufficiently large $h$ we have $\tmix(\epsilon;x) < 6L^2 h$.
\end{claim}
\begin{proof}
Let $G'$ denote the induced subgraph on the bottom $h+1$ levels of the graph (i.e., levels $2h+2,\ldots,3h+2$).
Since $|G'| = (1-o(1))|G|$ it clearly suffices to show that $(S_t)$ mixes within total-variation distance $\epsilon$ on $G'$.

By Claim~\ref{clm-tau-ell} (taking the worst case $\alpha=1$ corresponding to $\psi(x)=h$) with high probability we have
\[\tau_\ell \leq (\tfrac56+o(1))[L(5L-3)+2]h < 5L^2 h \deq t_1\,,\]
where the above strict inequality holds for any sufficiently large $h$ (as $L\geq 1$).

Recall that $\iso(G') \geq \kappa = (\iso(H_2)\wedge 1)/3$ by Lemma~\ref{lem-exp}, where $H_2$
is the explicit $4$-regular expander with second largest (in absolute value) eigenvalue $\lambda(H_2)$.
Further consider the graph $G''$ obtained by adding to $G'$ a perfect matching on level $2h+2$, thus making it $5$-regular. Clearly, adding edges can only increase the Cheeger constant and so $\iso(G'') \geq \kappa$ as well. Moreover, the discrete form of Cheeger's inequality (\cites{Alon,AM,Dodziuk,JS}), which for a $d$-regular expander $H$ with second largest (in absolute value) eigenvalue $\lambda$ states that
\begin{equation*}
\frac{d-\lambda}2 \leq \iso(H) \leq \sqrt{2d(d-\lambda)}\,,
\end{equation*}
here gives the following:
\begin{align*}
\frac{4-\lambda(H_2)}6 \;\wedge\; \frac13\leq \kappa \leq \iso(G'') \leq \sqrt{10(5-\lambda(G''))}\,.
\end{align*}
In particular we obtain that $\gamma \deq 1-\frac{\lambda(G'')}5$ satisfies
\begin{equation}
  \label{eq-gamma-eq}
  \gamma > \bigg(\frac{4-(\lambda(H_2)\;\vee\; 2)}{6\sqrt{50}}\bigg)^2
\end{equation}
while a simple random walk $(X_t)$ on $G''$ is well-known to satisfy
\begin{align*}
  \max_{x_0}\left\| \P_{x_0}\left(X_t \in \cdot \right) - |G''|^{-1} \right\|_{2}
&\leq \sqrt{|G''|} \exp\left(-\gamma t \right)\,.
\end{align*}
As $|G''| \asymp 2^{6h}$ we infer that after
\[ \frac{3\log 2}{\gamma} h + \frac{O(\log(1/\epsilon))}{\gamma} < \frac9{4\gamma} h \deq t_2\]
(the strict inequality holding for large enough $h$) steps we have
\begin{align*}
\max_{x_0}\left\| \P_{x_0}\left(X_{t_2} \in \cdot \right) - |G''|^{-1} \right\|_\mathrm{TV}
\leq \epsilon / 2\,.
\end{align*}
Recall that our choice of $t_1$ is such that $\tau_\ell < t_1$ w.h.p., i.e.\ the \SRW\ reaches level $3h+2$ by that time, and thereafter (due to its bias towards the leaves) it does not revisit level $2h+2$
until time $t_1+t_2$ except with a probability that is exponentially small in $h$. Since $G$ and $G''$ are identical on levels $2h+3,\ldots,3h+2$ we deduce that w.h.p.\ the \SRW\ performs at least $t_2$ consecutive steps in $G''$ following $\tau_\ell$. Altogether, for large enough $h$ we have
$\left\| \P_{x}\left(S_{t_1+t_2} \in \cdot \right) - \pi \right\|_\mathrm{TV} < \epsilon$ and so
$ \tmix(\epsilon;x) < t_1+t_2 $.

Finally, bearing~\eqref{eq-gamma-eq} and the choice of $L$ in ~\eqref{eq-L-choice} in mind,
\begin{equation}
  \label{eq-L-cond2} L \geq \frac{64}{4 - (\lambda(H_2)\;\vee\;2)} > \frac{\sqrt{\frac94}\cdot 6\sqrt{50}}{4 - (\lambda(H_2)\;\vee\;2)}
  \,,
\end{equation}
hence $L^2 > 9 /4\gamma$ and so $t_2 \leq L^2 h$ and $t_1+t_2 \leq 6L^2h$, as required.
\end{proof}

We now claim that the worst-case mixing time within any $0 < \epsilon < 1$ is attained by an initial vertex at distance $o(h)$ from the root.
Fix $0 <\epsilon < 1$, let $x$ be the initial vertex maximizing $\tmix(\epsilon;x)$ and recall that $\psi(x)$ denotes its level in the graph.
The combination of Claim~\ref{clm-tau-ell} and Corollary~\ref{cor-mixing-top} ensures that if $\psi(x) \leq h$ then necessarily $\psi(x) = o(h)$, in which case
\begin{equation}
  \label{eq-worst-near-root} \tmix(\epsilon;x) = \left(\tfrac53+o(1)\right)(5L^2-3L+1)h \deq t^\star\,.
\end{equation}
An immediate consequence of the requirement~\ref{eq-L-cond2} on $L$ is that $L \geq 50$, hence $t^\star > 8L^2 h$ for any sufficiently large $h$.
Therefore, we cannot have $\psi(x) > h$ since by Claim~\ref{clm-mix-bottom} that would imply that $\tmix(\epsilon;x)<6L^2 h$ contradicting the fact that $x$ achieves the worst-case mixing time.

Overall we deduce that for any $0 < \epsilon < 1$ we have \[\tmix(\epsilon) = \max_{x}\tmix(\epsilon;x) = (1+o(1))t^\star\,,\] thus confirming that the \SRW\ on the above constructed family of $5$-regular expanders exhibits total-variation cutoff from a worst starting location.

\bigskip

It remains to describe how our construction can be (relatively easily) modified to be $3$-regular rather than $5$-regular.

The immediate step is to use binary trees instead of $4$-ary trees, after which we are left with the problem of embedding the explicit expanders $H_1$ and $H_2$ without increasing the degree. This will be achieved via the line-graphs of these expanders, hence our explicit expanders will now have slightly different parameters:
\begin{itemize}
  \item $H_1$ : An explicit $3$-regular expander on $2^{h+1}$ vertices.
  \item $H_2$ : An explicit $3$-regular expander on $2^{3h+1}$ vertices.
\end{itemize}
Recall that given a tree rooted at some vertex $u$, denoted by $\cT_u$, its edge-stretched version is obtained by replacing each edge by a $2$-path of length $L$, and the collection of all new interior vertices (due to subdivision of edges) is denoted by $\cT_u^*$.
The modified construction is as follows:
\begin{enumerate}[\quad 1.]
\item Levels 0,1,2: First levels of a binary tree.
\begin{itemize}
  \item Denote by $U=\{u_1,\ldots,u_6\}$ the vertices in level $2$.
\end{itemize}
\item Levels $3,\ldots,h+2$: Stretched binary trees rooted at $U$:
\begin{itemize}
  \item Connect vertices from $\cT_{u_i}^*$ (interior vertices along $2$-paths) to the corresponding (isomorphic) vertices in $\cT^*_{u_{2i}}$, i.e.\ inter-connect the interior vertices via perfect matchings.
  \item Denote by $A$ the $6\cdot2^h$ vertices in level $h+2$.
\end{itemize}
\item Levels $h+3,\ldots,2h+2$: Edge-stretched binary trees rooted at $A$ and inter-connected via the line-graph of $H_1$ using auxiliary vertices:
\begin{itemize}
  \item Associate each binary tree $\cT_{a_i}$ rooted at $A$ to an \emph{edge} of $H_1$.
  \item For each $x \in \cT_{a_i}^*$ (interior vertex on a $2$-path) we connect it to a new auxiliary vertex $x'$ and associate $x'$ with a unique edge of $H_1$.
  \item We say that $x\in\cT_{a_i}^*$ and $y\in\cT_{a_j}^*$ are isomorphic if the isomorphism from $\cT_{a_i}$ to $\cT_{a_j}$ maps $x$ to $y$. Add $|H_1|$ new auxiliary vertices per equivalence class of $|A|$ such isomorphic vertices, identify them with the vertices of $H_1$ and connect every new vertex $v$ to the auxiliary vertices $x',y',z'$ representing the edges incident to it.
\end{itemize}
\item Levels $2h+3,\ldots,3h+2$: A forest of binary trees.
\item Last level: leaves are inter-connected via the line-graph of $H_2$:
\begin{itemize}
  \item Associate the $6\cdot 2^{3h}$ vertices with the edges of $H_2$.
  \item Add $|H_2|$ new auxiliary vertices, each connected to the leaves corresponding to edges that are incident to it in $H_2$.
\end{itemize}
\end{enumerate}
It is easy to verify that the walk along the cross-edges of the $\cT_{a_i}^*$'s now corresponds to a lazy (unbiased) random walk on the edges of $H_1$. Similarly, the walk along the cross-edges connecting the leaves corresponds to the \SRW\ on the edges of $H_2$. Hence, all of the original arguments remain valid in this modified setting for an appropriately chosen fixed $L$.
This completes the proof of Theorem~\ref{thm-1}. \qed

\subsection{Explicit expanders without total-variation cutoff}\label{sec-nocutoff}

The explicit cubic expanders with cutoff constructed in the previous section (illustrated in Fig.~\ref{fig-expcons}) can be easily modified so that the \SRW\ on them from a worst starting position would \emph{not} exhibit total-variation cutoff.

To do so, recall that in the above-described family of graphs, each vertex of the subset $U$ was the root of a regular tree of height $h$ whose edges were stretched into $L$-long $2$-paths (see Item~\ref{cons-part-1} of the construction).
We now tweak this construction by stretching some of the edges into $2$-paths of length $L'$.
Namely, for subtrees rooted at the \emph{odd} vertices in level $h/2$ of these trees we stretch the edges into paths of length $L' > L$. Under this modified stretching the trees $\cT_{u_i}$ are clearly still isomorphic, hence the cross edges are inter-connecting $2$-paths of the same lengths.

By the arguments above, starting from any level $\ell > h/2$ the mixing is faster compared to the root, and if $L'/L$ is sufficiently small then the root remains the asymptotically worst starting position. However, starting from the root (and in fact, starting from any level $\ell \leq h/2$) the hitting time to the set $A$ is no longer concentrated due to the odd/even choice of subtree at level $h/2$. Therefore, from the worst starting position we have that the hitting time to the leaves is concentrated on two distinct values (differing by a fixed multiplicative constant), each with probability $\frac12-o(1)$. This implies that the ratio $\tmix(\frac14)/ \tmix(\frac34)$ is bounded away from $1$ and in particular this explicit family of expanders does not have total-variation cutoff.

\subsection{Proof of Theorem~\ref{thm-2}: cutoff at any prescribed location order}\label{sec-prescribed}

Suppose $H$ is an explicit $3$-regular expander on $m$ vertices provided by Theorem~\ref{thm-1}, and recall that the \SRW\ on this graph exhibits cutoff at $C \log m$ where $C>0$ is some absolute constant. Our graph $G$ will be the result of replacing every edge of $H$ by the $3$-regular analogue of a $2$-path, which we refer to as a ``cylinder'',
illustrated in Fig.~\ref{fig-expconsgen}. The length of each cylinder is set to be $L=L(m)$ satisfying $L\equiv 1\pmod{4}$. Notice that the total number of vertices in $G$ is
\begin{equation}\label{e:numberOfVertices}
n = |V(H)|+|E(H)|\tfrac32(L-1) = \left(1 + \tfrac94(L-1)\right)m\,.
\end{equation}
Since $m\to\infty$ and the \SRW\ on $H$, started at a worst-case starting position, traverses $(C+o(1))\log m$ edges until mixing, we infer from CLT, as well as the fact that the expected passage-time through an $L$-long cylinder is $L^2$, that the analogous random walk on $G$ has cutoff at
\begin{equation}\label{e:mixStrechted}
\tmix = C L^2 \log m = (C+o(1))L^2 \log(n/L)\,.
\end{equation}
When $L(m)=O(1)$ we have $\tmix \asymp \log n$. On the other extreme end, when $L$ grows arbitrarily fast as a function of $m$ we obtain that it approaches $n$ arbitrarily closely but we must still having $L = o(n)$ since $n/L \asymp m \to \infty$. In that case $\tmix$ approaches $n^2$ arbitrarily closely while having a strictly smaller order.

To complete the construction it remains to observe that we may choose $L$ and $m$ so that $|G_n|\asymp n$ and $\tmix \asymp t_n$.  This can be achieved by first selecting $L$ so that $t_n \asymp (C+o(1))L^2 \log(n/L)$ and then selecting a graph constructed through Theorem~\ref{thm-1} on $m$ vertices for some $m\asymp n/L$ (note that in the theorem we construct graphs of size essentially $(c +o(1)) 2^{3h}$ so this is always possible).

\begin{figure}
\centering \includegraphics[width=4.5in]{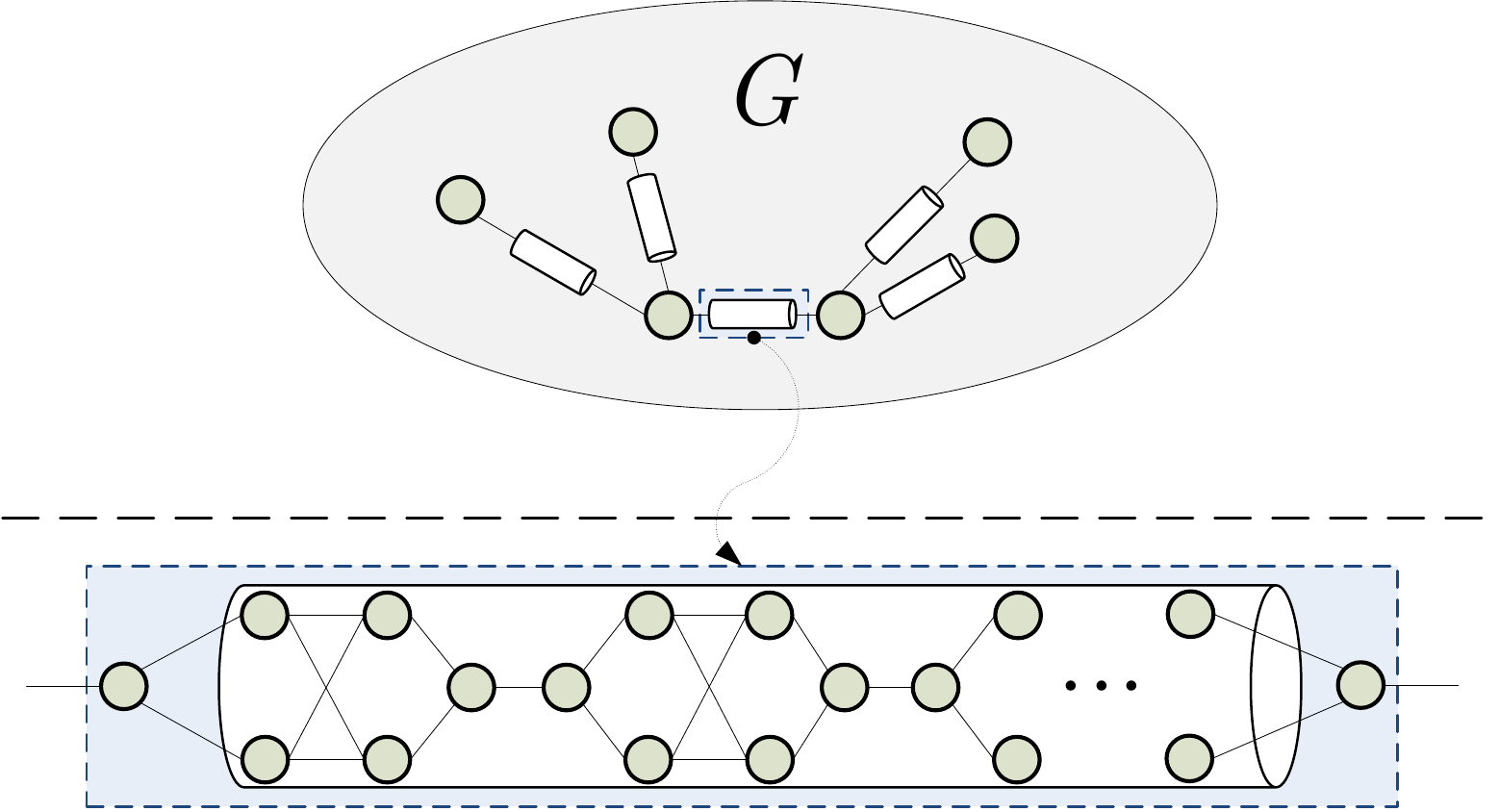}
\caption{Extension of the explicit expanders with cutoff to graphs with prescribed order of cutoff location.}
\label{fig-expconsgen}
\end{figure}

To show that there is no cutoff whenever $\tmix\asymp n^2$ we will argue that in that case we have $\gap = O(n^{-2})$, in contrast to the necessary condition for cutoff $\gap^{-1}=o(\tmix)$ due to Peres (cf.~\cite{Peres}), discussed in the introduction.
\begin{lemma}\label{lem-trel-no-cutoff}
   Let $G$ be a graph on $n$ vertices with degrees bounded by some $\Delta$ fixed on which the \emph{\SRW} has $\tmix \asymp n^2$. Then the spectral-gap of the walk satisfies $\gap\asymp n^{-2}$. In particular, the \emph{\SRW} on $G$ does not exhibit cutoff.
\end{lemma}
\begin{proof}
Observe that the above graph must satisfy $\diam(G) \geq c n$ for some fixed $c > 0$ as it is well-known (cf., e.g., \cite{AF}) that the lazy walk on any graph $H$ has $\tmix = O(\diam(H) \vol(H))$ and in our case $\vol(G) \leq \Delta n = O(n)$.
%

Let $x,y\in V(G)$ be two vertices whose distance in $G$ is
\[N \deq \dist_G(x,y) = \diam(G) \geq c n\,.\]
Our lower bound on the gap will be derived from its representation via the Dirichlet form,
according to which
\begin{align}\label{eq-dirichlet-form}
\gap = \inf_f \frac{\mathcal{E}(f)}
{\var(f)} = \inf_f\frac{\frac{1}{2}\sum_{x,y\in\Omega}\left[f(x)-f(y)\right]^2\pi(x)P(x,y)}{\var_\pi f}\,,
\end{align}
where $\pi$ is the (uniform) stationary measure and $P$ is the transition kernel of the \SRW.
As a test-function $f:V(G)\to\R$ in the above form choose
\[ f(v) \deq \dist_G(x,v)\,.\]
Clearly we have $\mathcal{E}(f) \leq 1$ and a lower bound of order $n^2$ on the variance follows from the fact that two sets of linear size each have a linear discrepancy according to $f$. Namely,
\[ \pi\left(f^{-1}(\{0,\ldots,\lfloor N/4\rfloor\})\right) \geq \frac{c}4~,~\pi\left(f^{-1}(\{\lceil3N/4\rceil,\ldots,N\})\right) \geq \frac{c}4\,,\]
as a result of which
\[ \var(f) \geq (c/4)(N/4) ^2 > c' n^2\,.\]
We conclude that $\gap = O(n^{-2})$, thus completing the proofs of Lemma~\ref{lem-trel-no-cutoff} and Theorem~\ref{thm-2}.
\end{proof}

\section{Concluding remarks and open problems}
\begin{list}{\labelitemi}{\leftmargin=2em}
\item Recent results in~\cite{LS} showed that almost every regular expander graph has total-variation cutoff (prior to that there were no known examples for bounded-degree graphs with this phenomenon); here we provided a \emph{first explicit} construction for bounded-degree expanders with cutoff.

\item The expanders constructed in this work are non-transitive. Moreover, our proof exploits their highly asymmetric structure in order to control the mixing time of the random walk from various starting locations.
    It would be interesting to obtain an explicit construction of transitive expanders with total-variation cutoff.

\item A slight variant of our construction gives an example of a family of expanders where the \SRW\ does \emph{not} exhibit cutoff, thereby disagreeing with Peres' cutoff-criterion. Both here and in another such example due to Peres and Wilson~\cite{PW} the expanders are non-transitive (hence the restriction to transitive graphs in Peres' conjecture stated next).

\item While it is conjectured by Peres that the random walk on any family of transitive bounded-degree expanders exhibits total-variation cutoff, there is not even a single example of such a transitive family where cutoff was proved (or disproved).

\item For general (not necessarily expanding) bounded-degree graphs on $n$ vertices it is well-known that $\tmix = O(n^2)$. Here we showed that cutoff can occur essentially anywhere up to $o(n^2)$ by constructing cubic graphs with cutoff at any such prescribed location. Furthermore, this is tight as we prove that if $\tmix \asymp n^2 $ then cutoff cannot occur.

\end{list}

\subsection*{Acknowledgment}
We would like to thank Yuval Peres for suggesting the problem, his encouragement and useful discussions.

\end{document}